\documentclass[11pt,a4paper]{amsart}
\usepackage[all]{xy}
\usepackage{graphicx}
\pdfoutput=1

\numberwithin{equation}{section}

\newcommand{\ie}{{\em i.e. }}
\newcommand{\cf}{{\em cf.}\ }

\newtheorem{theorem}{Theorem}[subsection]

\newtheorem{lemma}[theorem]{Lemma}
\newtheorem{proposition}[theorem]{Proposition}

\newtheorem{observation}[theorem]{Observation}
\newtheorem{lemma and definition}[theorem]{Lemma and Definition}

\theoremstyle{definition}

\newtheorem{notation}[theorem]{Notation}
\newtheorem{remark}[theorem]{Remark}

\newtheorem{definition}[theorem]{Definition}

\newcommand{\Z}{\mathbb{Z}}
\newcommand{\T}{\mathbb{T}}

\newcommand{\Q}{\mathbb{Q}}
\newcommand{\Qext}{\mathbb{Q}_{\infty}}

\renewcommand{\P}{\mathbb{P}}

\newcommand{\px}{\phantom{x}}

\newcommand{\ca}{{\mathcal A}}

\newcommand{\cF}{{\mathcal F}}

\begin{document}
\title{ON THE C-VECTORS and G-VECTORS OF THE MARKOV CLUSTER ALGEBRA}
\author{ALFREDO N\'AJERA CH\'AVEZ}
\maketitle

\begin{abstract}
We describe the $\mathbf{c}$-vectors and $\mathbf{g}$-vectors of the Markov cluster algebra in terms of a special family of triples of rational numbers, which we call the Farey triples.
\end{abstract}

\section*{Introduction}
In papers such as \cite{Clusters 3, QP 1, Labardini 1, Labardini 2, Plamondon} (to mention just a few), the authors present (as examples and propositions) interesting properties of the cluster algebra arising from the quiver of Figure 1, as well as of the potentials, Jacobian algebras and cluster categories associated to it.
\vspace{-1mm}
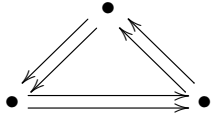
\begin{figure}[htbp]
\label{Double}
\begin{equation*}
\xymatrix{
 & \bullet \ar@<-0.5ex>[ld] \ar@<0.5ex>[ld] &  \\
\bullet \ar@<-0.5ex>[rr] \ar@<0.5ex>[rr] & & \bullet \ar@<-0.5ex>[lu] \ar@<0.5ex>[lu]
}
\vspace{-1mm}
\end{equation*}
\caption{Double cyclic triangle}
\vspace{-1.5mm}
\end{figure}

The cluster algebra arising from this quiver is frequently called the \emph{Markov cluster algebra}, in the theory of cluster algebras associated to triangulated surfaces (see \cite{FST}), it  can be regarded as the cluster algebra arising from a torus with one puncture. When this algebra has principal coefficients we denote it by $\ca_M$. In this paper, we describe the (extended) exchange matrices of $\ca_M$ and thus the set of its $\mathbf{c}$-vectors, which can be considered as a generalization of a root system. In \cite{Nakanishi} and \cite[Theorem 1.2]{Nakanishi Zelevinsky}, the authors show how the $\mathbf{c}$-vectors of a cluster algebra associated to a quiver are related to its $\mathbf{g}$-vectors. Thus, in our case, we can obtain the $\mathbf{g}$-vectors. Section 1 is a reminder on cluster algebras with principal coefficients. In section 2 we introduce the Farey triples, which are triples of rational numbers satisfying a specific arithmetic condition. We define a mutation operation on Farey triples and show that the resulting \emph{exchange graph} is a $3$-regular tree $\T_3$. In Section 3 we associate to each Farey triple $T$ an exchange matrix $M_T$ of $\ca_M$, and write the entries of $M_T$ in terms of the components of $T$. 
\section*{Acknowledgments}
I would like to express my sincere thanks to my master advisor, Professor Christof Geiss Hahn, for his support during my studies and for introducing me to the theory of cluster algebras. I am grateful to my current advisor, Professor Bernhard Keller, for helpful comments on previous versions of this article.

\section{Preliminaries on cluster algebras}
\subsection{Cluster algebras with principal coefficients} In this section, we recall the construction of cluster algebras with principal coefficients from \cite{Clusters 4}. For an integer $x$, we use the notations
\begin{equation*}
[x]_+=\max(x,0)
\end{equation*}
and
\begin{equation*}
\text{sgn}(x)=
\begin{cases}
-1 & \text{if  } x<0 \\
\ \ 0 & \text{if  } x=0\\
\ \ 1 & \text{if  } x>0
\end{cases}
\end{equation*}


\begin{definition}
The \emph{tropical semifield} on a finite family of variables $u_j,\ j \in J$, is the abelian group (written multiplicatively) $\text{Trop}(u_j: j \in J)$ freely generated by the $u_j,\ j \in J$. It is endowed with an auxiliary addition~$\oplus$ defined by
\begin{equation}
\label{eq:tropical-addition}
\prod_j u_j^{a_j} \oplus \prod_j u_j^{b_j} =
\prod_j u_j^{\min (a_j, b_j)} \,.
\end{equation}
\end{definition}

\begin{remark}
From now on, we let $n$ be a positive integer and $\P$ be the tropical semifield on the indeterminates $x_{n+1}, \ldots , x_{2n}$. Notice that $\Q \P$, the group algebra on the abelian group $\P$, is naturally identified with the algebra of Laurent polynomials with rational coefficients in the variables $x_{n+1}, \ldots , x_{2n}$. We denote by $\cF$ the field of fractions of the ring of polynomials with coefficients in $\Q\P$ in $n$ indeterminates.
\end{remark}

\begin{definition}
A \emph{seed} in $\cF$ is a pair $(\tilde{B},\mathbf{x})$ formed by
\begin{itemize}
\item a $2n \times n$ integer matrix $\tilde{B}=(b_{ij})$, whose principal part $B$, \ie the top $n \times n$ submatrix, is skew-symmetric (equivalently, associated to a quiver $\Gamma$ with no oriented cycles of length $\leq 2$ via the formula $b_{ij}$=$|$number of arrows $i\rightarrow j $ in $ \Gamma$$|-|$number of arrows $j\rightarrow i$ in $\Gamma|$ for $1 \leq i,j\leq n$);
\item a free generating set $\mathbf{x}=\lbrace x_1,\ldots , x_n\rbrace $ of the field $\cF$.
\end{itemize}
The matrix $\tilde{B}$ is called the \emph{(extended) exchange matrix} and the set $\mathbf{x}$ the \emph{cluster} of the seed $(\tilde{B},\mathbf{x})$.
\end{definition}

\begin{definition}
Let $(\tilde{B},\mathbf{x})$ be a seed in $\cF$, and let $1\leq k \leq n$ be an integer. The \emph{seed mutation} $\mu_k$ in the direction $k$ transforms $(\tilde{B},\mathbf{x})$ into the seed $\mu_k(\tilde{B},\mathbf{x})=(\tilde{B}',\mathbf{x}')$ defined as follows:
\begin{itemize}
\item The entries of $\tilde{B}'=(b'_{ij})$ are given by
\begin{equation}
b'_{ij}=
\begin{cases}
-b_{ij} & \text{if  } i=k\text{ or } j=k; \\
b_{ij}+\text{sgn}(b_{ik})[b_{ik}b_{kj}]_+ & \text{otherwise}.
\end{cases}
\end{equation}
\item The cluster $\mathbf{x}'=\lbrace x'_1,\ldots, x'_n \rbrace$ is given by $x'_j=x_j$ for $j\neq k$, whereas $x'_k\in \cF$ is determined by the \emph{exchange relation}
\begin{equation}
\label{exchange relation}
x'_kx_k=\prod_{i=1}^{2n}x_i^{[b_{ik}]_+}+\prod_{i=1}^{2n}x_i^{[-b_{ik}]_+}
\end{equation}
\end{itemize}
\end{definition}
\begin{remark}
Mutation in a fixed direction is an involution, \ie $\mu_k\mu_k(\tilde{B},\mathbf{x})=(\tilde{B},\mathbf{x})$.
\end{remark}

\begin{definition}
Let	$\T_n$ be	the $n$-regular tree,	whose edges are labeled by the numbers $1, \ldots , n$ so that the $n$ edges emanating from each vertex carry different labels. A \emph{cluster pattern} is the assignment of a seed $(\tilde{B}_t,\mathbf{x}_t)$ to each vertex $t$ of $\T_n$ such that the seeds assigned to vertices $t$ and $t'$, linked by an edge with label $k$, are obtained from each other by the seed mutation $\mu_k$. Fix a vertex $t_0$ of the $n$-regular tree $\T_n$. Clearly, a cluster pattern is uniquely determined by the initial seed $(\tilde{B}_{t_0} , x_{t_0} )$, which can be chosen arbitrarily. The $\mathbf{c}$\emph{-vectors} of a seed $(\tilde{B}_t , \mathbf{x}_t )$ are elements of the set $\lbrace \mathbf{c}_{j;t}=(b_{n+1j}^{t},\ldots ,b_{2nj}^{t})\in\Z^{n}: 1\leq j\leq n \rbrace$, where $\tilde{B}_t=(b^t_{ij})$.
\end{definition}

\begin{definition}
Fix a seed  $(\tilde{B},\mathbf{x})$ and let  $(\tilde{B}_t,\mathbf{x}_t),\ t \in \T_n$, be the unique cluster pattern with initial seed  $(\tilde{B},\mathbf{x})$. The \emph{clusters} associated with  $(\tilde{B},\mathbf{x})$ are the sets $\mathbf{x}_t,\ t \in \T_n$. The \emph{cluster variables} are the elements of the clusters. 
The \emph{cluster algebra} $\ca(\tilde{B}) = \ca(\tilde{B},\mathbf{x})$ is the $\Z\P$-subalgebra of $\cF$ generated by the cluster variables. Its \emph{ring of coefficients} is $\Z\P$. We say that the cluster algebra $\ca(\tilde{B})$ has \emph{principal coefficients} at the vertex $t_0$ if the complementary part of $\tilde{B}_{t_0}$ (\ie the bottom $n\times n$ submatrix) is the $n\times n$ identity matrix.
\end{definition}




\subsection{The $\mathbf{g}$-vectors} In this section we recall the natural $\Z^{n}$-grading in a cluster algebra with principal coefficients. This leads us to the definition of the $\mathbf{g}$-vectors.
\begin{definition}
Let $\ca(\tilde{B},\mathbf{x})$ be the cluster algebra with principal coefficients at a vertex $t_0$,  defined by the initial seed $(\tilde{B},\mathbf{x})=((b_{ij}),\lbrace x_1,\ldots,x_n \rbrace)$. The \emph{(natural) $\Z^n$-grading} of the ring $\Z[x_1{\pm 1},\ldots , x_n^{\pm 1};x_{n+1},\ldots , x_{2n}]$ is given by
\begin{equation}
\label{grading}
\deg(x_j)=
\begin{cases}
\mathbf{e}_j  & \text{if } 1\leq j \leq n\\
\mathbf{-b}_j  & \text{if } n+1\leq j \leq 2n
\end{cases}
\end{equation}
where $\mathbf{e}_1,\ldots, \mathbf{e}_n$ are the standard basis (column) vector in $\Z^{n}$, and $\mathbf{b}_j=\Sigma_i b_{ij}\mathbf{e}_i$.
\end{definition}

\begin{theorem} (\cite[Corollary 6.3]{Clusters 4}) Under the $\Z^n$-grading given by (\ref{grading}), the cluster algebra $\ca(\tilde{B},\mathbf{x})$ is a $\Z^n$-graded subalgebra of $\Z[x_1{\pm 1},\ldots , x_n^{\pm 1};x_{n+1},\ldots , x_{2n}]$.
\end{theorem}

\begin{definition}
By iterating the exchange relation (\ref{exchange relation}) we can express every cluster variable $x_{j;t}$ as a (unique) rational function in $x_1,\ldots, x_{2n}$ given by a subtraction-free rational expression; we denote this rational function by $X_{j;t}$. The \emph{$\mathbf{g}$}-vectors of a seed $(\tilde{B}_t , \mathbf{x}_t$ are:
\begin{equation}
\mathbf{g}_{j;t}=\left[
\begin{array}{c}
g_1 \\
\vdots \\
g_n
\end{array}
\right]
=\deg(X_{j;t}),
\end{equation}
we call a matrix of the form $(\mathbf{g}_{1;t}, \ldots , \mathbf{g}_{n;t})$ a $\mathbf{g}$-matrix.
\end{definition}

\section{Farey triples}
\subsection{The Farey sum} In this section we introduce some basic properties of the Farey triples, which are also useful in the study of tubular cluster algebras (see \cite{BG,BGJ}). Let $\Qext = \Q \cup \lbrace \infty \rbrace $ be the totally ordered set of extended rational numbers. 
\begin{notation}
For $q \in \Qext$, we denote by $d(q)$ and $r(q)$ the integers such that:
\begin{itemize}
\item[$\bullet$] $q=\dfrac{d(q)}{r(q)}$,
\item[$\bullet$] $\gcd(d(q),r(q))=1$,
\item[$\bullet$] $r(q)\geq 0$.
\end{itemize}
It is obvious that $d(q)$ and $r(q)$ are uniquely determined. In particular, $0=\frac{0}{1}$ and $\infty=\frac{1}{0}$.
\end{notation}

\begin{definition}
Given $q,q' \in \Qext $ we define
\begin{equation}
\Delta (q,q'):=\left|
\det
\left(
\begin{array}{cc}
d(q) & d(q')\\
r(q) & r(q')
\end{array}
\right)
\right|
\end{equation}
In case that $\Delta(q,q')=1$, we call $q$ and $q'$ \emph{Farey neighbors}. 
\end{definition}

\begin{definition}
Let $q,q' \in \Qext $ be Farey neighbors. The \emph{Farey sum} of $q$ and $q'$ is
\begin{equation}
q \oplus q':= \dfrac{d(q)+d(q')}{r(q)+r(q')},
\end{equation}
and their \emph{Farey difference} is
\begin{equation}
q \ominus q':= \dfrac{d(q)-d(q')}{r(q)-r(q')}.
\end{equation}
Note that both of these operations are commutative.
\end{definition}


\begin{definition}
\label{def: Farey triple}
We call $(q_1,q_2,q_3) \in \Qext^3$ a \emph{triple of neighbors} provided that $\Delta(q_i,q_j)=1$ for $1\leq i < j \leq 3$. A \emph{Farey triple} $[q_1,q_2,q_3]$ is a triple of neighbors considered up to permutation of its components. We use the indices $\lbrace f,s,t \rbrace=\lbrace 1,2,3 \rbrace$ for a Farey triple to indicate that $q_f<q_s<q_t$. 
\end{definition}  

\begin{lemma}
\label{s=f+t}
Let $(q_1,q_2,q_3)\in \Qext^3$ and $q_1<q_2<q_3$. Then $(q_1,q_2,q_3)$ is a triple of neighbors if and only if $q_2=q_1\oplus q_3$, \ie the Farey triples are essentially of the form $[q_f,q_f\oplus q_t,q_t]$. 
\end{lemma}

\begin{proof}
Fix $q_1$ and $q_3$ and consider the following system of equations in $d(q_2)$ and $r(q_2)$:
\begin{equation*}
\begin{array}{rl}
d(q_3)r(q_2)-r(q_3)d(q_2) & =1\\
-d(q_1)r(q_2)+r(q_1)d(q_2) & =1,
\end{array}
\end{equation*}
the determinant of the system is $\Delta(q_1,q_3)=1$, so it has a unique solution which is $d(q_2)=d(q_1)+d(q_3)$ and $r(q_2)=r(q_1)+r(q_3)$.
\end{proof}


\begin{lemma and definition}
\label{Farey decomposition}
For all $q \in \Q$, there exist two unique Farey neighbors $q',q''\in \Qext$ such that $q=q'\oplus q''$. We call such expression the Farey decomposition of $q$.
\end{lemma and definition}

\begin{proof}
Consider the diophantine equation $d(q)y-r(q)x=1$. Let $x=a$, $y=b$ be a solution of the equation with $0 \leq b < r(q) $. Put $q'=\frac{a}{b}$ and $q''=q \ominus q'$, then $q=q'\oplus q''$. To see the uniqueness of $q'$ and $q''$, suppose that $q=p \oplus p'$ is another decomposition. Without loss of generality we can assume that $p,q'<q<p'q''$. Since $\Delta(q,p)=\Delta(q,q')=1$ and $p,q'<q$, it follows that $d(q)(r(q')-r(p))=r(q)(d(q')-d(p))$. In particular $q' \neq p$ if and only if $r(q') \neq r(p)$. Suppose that $q' \neq p$. Then we have
\begin{equation*}
\dfrac{d(q)}{r(q)}=\dfrac{d(q')-d(p)}{r(q')-r(p)},
\end{equation*}
but $-r(q)<r(q')-r(p)<r(q)$, a contradiction. Hence $q'=p$. We do the same for $q''=p'$.
\end{proof}

\begin{definition}
Let $[q_f,q_s,q_t]$ be a Farey triple (with $q_f<q_s<q_t$ see \ref{def: Farey triple}) and $k\in \lbrace f,s,t \rbrace$. The \emph{mutation} $\mu_k$ in direction $k$ of $[q_f,q_s,q_t]$ is the Farey triple $\mu_k[q_f,q_s,q_t]$ defined as follows:
\begin{equation}
\mu_k[q_f,q_s,q_t]=\begin{cases}
[q_s,q_s \oplus q_t,q_t] & \text{if  }  k=f \\
[q_f,q_f \ominus q_t,q_t] & \text{if  }  k=s \\
[q_f,q_f \oplus q_s,q_s] & \text{if  }  k=t 
\end{cases}
\end{equation}  
In view of Lemma \ref{Farey decomposition} we have that the mutation is an involution, \ie we have $\mu_k\mu_k[q_f,q_s,q_t]=[q_f,q_s,q_t]$ for each $k$.
\end{definition}

\begin{definition}
The \emph{exchange graph} of the Farey triples is the graph whose vertices are the Farey triples, and whose edges connect triples related by a single mutation.  
\end{definition}

\begin{proposition}
The exchange graph of the Farey triples is a $3$-regular tree  $\T_3$. 
\end{proposition}
\begin{proof}
First we show that the graph is connected. It is sufficient to show that any Farey triple is connected with the triple $[\frac{-1}{1}, \frac{0}{1}, \frac{1}{0}]$. Call this triple the \emph{initial triple} and all others \emph{non-initial triples}. For any Farey triple  $[q_f,q_s,q_t]$, define its \emph{complexity} as $c[q_f,q_s,q_t]=|d(q_s)|+r(q_s)$. Since $c[q_f,q_s,q_t]=1$ implies that $[q_f,q_s,q_t]=[\frac{-1}{1}, \frac{0}{1}, \frac{1}{0}]$, it will be sufficient to show that  for any non-initial triple, there is a (unique) direction $k\in \lbrace f,s,t \rbrace$ such that $c(\mu_k[\frac{-1}{1}, \frac{0}{1}, \frac{1}{0}])<c[q_f,q_s,q_t]$. If $c[q_f,q_s,q_t]=[\frac{-z-1}{1}, \frac{-z}{1}, \frac{1}{0}]$ with $z\geq1$, then take $k=f$. In any other case we have either $q_i\geq0$ or $q_i\leq0$  for all $i$, so $k=s$ works.
It only remains to show that the graph does not contain cycles. Since the complexity of a triple always changes under mutation, the graph does not contains loops. Now, notice that if the graph contain an $n$-cycles with $n \geq 2$, then we can find a vertex of the cycle for which the complexity decreases in two different directions, a contradiction.
\end{proof}

\subsection{Some isomorphisms of $\T_3\setminus [\frac{0}{1}, \frac{-1}{1}, \frac{1}{0}]$} In the rest of this note, we write $ \mu_0$, $\mu_{ -1}$ and $ \mu_{\infty} $ to denote the mutations in direction of the fractions of the form $\frac{even}{odd}$, $\frac{odd}{odd}$ and $\frac{odd}{even}$, respectively. We label the edges of $\T_3$ with the same indices, and denote by $\T_3^i$ the connected component of $\T_3 \setminus [\frac{0}{1}, \frac{-1}{1}, \frac{1}{0}]$ having the triple $\mu_i[\frac{0}{1}, \frac{-1}{1}, \frac{1}{0}]$ as a vertex.

\begin{remark}
\label{parity}
From now on, even if the Farey triples are considered up to permutation of their components, when we write a Farey triple, we assume that it is written in the order $[\frac{even}{odd}$, $\frac{odd}{odd}$,$ \frac{odd}{even}]$.  
\end{remark}

\begin{definition}
\label{iso of components}
We define a graph isomorphism $\phi:\T_3^{-1} \rightarrow \T_3^{\infty}$ as follows:
\begin{itemize}
\item[$\bullet$] The triple $[\frac{0}{1}, \frac{1}{1}, \frac{1}{0}]=\mu_{-1}[\frac{0}{1}, \frac{-1}{1}, \frac{1}{0}]$ is mapped to the triple $[\frac{0}{1}, \frac{-1}{1}, \frac{-1}{2}]=\mu_{\infty}[\frac{0}{1}, \frac{-1}{1}, \frac{1}{0}]$, 
\item[$\bullet$] The edges with label $-1,\ 0,\ \infty$ are mapped to the edges with label $\infty,\ -1,\ 0$, resp.
\end{itemize}
Analogously, we define a graph isomorphism $\psi:\T_3^{-1} \rightarrow \T_3^0$ as follows:
\begin{itemize}
\item[$\bullet$] The triple $[\frac{0}{1}, \frac{1}{1}, \frac{1}{0}]$ is mapped to the triple $[\frac{-2}{1}, \frac{-1}{1}, \frac{1}{0}]=\mu_{0}[\frac{0}{1}, \frac{-1}{1}, \frac{1}{0}]$, 
\item[$\bullet$] The edges with label $-1,\ 0,\ \infty$ are mapped to the edges with label $0,\ \infty,\ -1$, resp.
\end{itemize}
\end{definition}

\begin{proposition}
The isomorphisms of \ref{iso of components} can be expressed in terms of the following formulas:
\begin{eqnarray}
\phi[q_{0},q_{-1},q_{\infty}]& =& \left[ \frac{-r(q_{\infty})}{r(q_{\infty})+d(q_{\infty})},\frac{-r(q_0)}{r(q_0)+d(q_0)},\frac{-r(q_{-1})}{r(q_{-1})+d(q_{-1})}  \right],\\
\psi[q_{0},q_{-1},q_{\infty}]&=&\left[ \frac{-(r(q_{-1})+d(q_{-1}))}{d(q_{-1})},\frac{-(r(q_{\infty})+d(q_{\infty}))}{d(q_{\infty})},\frac{-(r(q_{0})+d(q_{0}))}{d(q_{0})}  \right].
\end{eqnarray}
\end{proposition}
\begin{proof}
We prove it only for $\phi$ using induction on $c[q_{0},q_{-1},q_{\infty}]$, for $\psi$ we can proceed in the same way. We see that the formula in Definition \ref{iso of components} for $\phi[\frac{0}{1}, \frac{1}{1}, \frac{1}{0}]$ satisfies the statement, so we take it as base for the induction. Without loos of generality, we assume that $q_0<q_{-1}<q_\infty$. Thus, the complexity of the triple increases only in directions $0$ and $\infty$. We consider only direction $0$, direction $\infty$ is similar. By induction we obtain:
\begin{eqnarray*}
\phi(\mu_0[q_{0},q_{-1},q_{\infty}])&=&\mu_{-1} \left[ \frac{-r(q_{\infty})}{r(q_{\infty})+d(q_{\infty})},\frac{-r(q_0)}{r(q_0)+d(q_0)},\frac{-r(q_{-1})}{r(q_{-1})+d(q_{-1})} \right]\\
&=&\left[\frac{-r(q_{\infty})}{r(q_{\infty})+d(q_{\infty})}, \frac{-(r(q_{\infty})+r(q_{-1}))}{r(q_{\infty})+d(q_{\infty})+r(q_{-1})+d(q_{-1})},\frac{-r(q_{-1})}{r(q_{-1})+d(q_{-1})}  \right].
\end{eqnarray*}
Since $q_{-1}\oplus q_{\infty}=\frac{d(q_{-1})+d(q_{\infty})}{r(q_{-1})+r(q_{\infty})}$, it is clear that the result follows for $\phi$.  
\end{proof}

\section{The exchange matrices of $\ca_M$}

\subsection{Description} In this section, we use the Farey triples in order to give an explicit description of the extended exchange matrices of the Markov cluster algebra (with principal coefficients) $\ca_M$. An important property of the quiver is that the mutation at any vertex only reverses the orientation of the arrows.
\begin{equation*}
\label{quiver}
\begin{array}{c c c}
\xymatrix{
 & -1 \ar@<-0.5ex>[ld] \ar@<0.5ex>[ld] &  \\
0 \ar@<-0.5ex>[rr] \ar@<0.5ex>[rr] & & \infty \ar@<-0.5ex>[lu] \ar@<0.5ex>[lu]
}
&
\xymatrix{
 \ar@<-3ex>[rr]^{\mu_k} & & \ar@<3.1ex>[ll]
 }

&
\xymatrix{
& -1 \ar@<-0.5ex>[dr] \ar@<0.5ex>[dr] &  \\
0 \ar@<-0.5ex>[ur] \ar@<0.5ex>[ur] & & \infty \ar@<-0.5ex>[ll] \ar@<0.5ex>[ll]
}

\end{array}
\end{equation*}
We denote by $B^+$ (resp. $B^-$) the matrix associated with the left hand (resp. right hand) quiver of \ref{quiver}. Then the initial matrix of $\ca_M$ is:
\begin{eqnarray*}
\label{initial matrix}
0 \px \px  -1\ \px  \px  \px \ \infty \px \px \px \\
\downarrow \px  \px  \px \downarrow \px  \px  \px \ \downarrow \px  \px  \px \\
\tilde{B}_0= \left(
\begin{array}{c c c}
0 & -2\phantom{-} & 2 \\
2 & 0 & -2\phantom{-} \\
-2\phantom{-} & 2 & 0 \\
1 & 0 & 0 \\
0 & 1 & 0 \\
0 & 0 & 1
\end{array} \right)
\end{eqnarray*}
We consider the following assignment defined recursively. To the initial Farey triple  $[\frac{0}{1}, \frac{-1}{1}, \frac{1}{0}]$, assign  the initial matrix $\tilde{B}_0$. To any  Farey triple $\mu_{a_n}\dots \mu_{a_1}[\frac{0}{1}, \frac{-1}{1}, \frac{1}{0}]$ with $a_i \in \lbrace 0,-1,\infty \rbrace$, assign the extended exchange matrix $\mu_{a_n}  \dots \mu_{a_1}(\tilde{B}_0)$. Note that this assignment is a surjection onto the set of the exchange matrices of $\ca_M$. We use the following notation to denote the assignment $[q_{0}, q_{-1} ,q_{\infty} ] \longrightarrow M$ (see Remark \ref{parity}):
\begin{eqnarray}
q_{0} \px \px q_{-1} \px \px q_{\infty} \px \px \nonumber \\ 
M= \left(
\begin{array}{c c c}
 &\ B^{\pm}& \\
a \px  & d   & \px g  \\
 b \px  & e   & \px h  \\
 c \px  & f   & \px i  
\end{array} \right)
\end{eqnarray}
Thus we may speak of mutating both, a triple and the associated matrix. Our goal is to describe all the extended exchange matrices using the Farey triples. First we describe the matrices obtained by alternating mutations in direction $-1$ and $0$. 
\begin{equation} 
\label{alternating} 
\begin{tabular}{c c c}
$\left(
\begin{tabular}{c c c}
1 & 0 & 0 \\
0 & 1 & 0 \\
0 & 0 & 1
\end{tabular} 
\right)$&
\xymatrix{
 \ar[rr]^{\mu_{-1}} & & 
 }
&
$\left(
\begin{tabular}{c c c}
1 & 0 & 0 \\
2 & -1 & 0 \\
0 & 0 & 1
\end{tabular} 
\right)$
\\
& &
\xymatrix{
& \ar[d]^{\mu_{0}} \\
&
 }
\\

$\left(
\begin{tabular}{c c c}
3 & -2 & 0 \\
4 & -3 & 0 \\
0 & 0 & 1
\end{tabular}
\right)$
&
\xymatrix{
  & & \ar[ll]_{\mu_{-1}}
 }
&
$\left(
\begin{tabular}{c c c}
-1 & 2 & 0 \\
-2 & 3 & 0 \\
0 & 0 & 1
\end{tabular}
\right)$
\\
\xymatrix{
& \ar[d]^{\mu_{0}} \\
&
 }
 & &\\
$\left(
\begin{tabular}{c c c}
-3 & 4 & 0 \\
-4 & 5 & 0 \\
0 & 0 & 1
\end{tabular}
\right)$ &
\xymatrix{
 \ar[rr]^{\mu_{-1}} & & 
 } 
 &
 $\ldots$
\end{tabular}
\end{equation}
These matrices correspond to triples of the form $[\frac{a}{1},\frac{a\pm1}{1},\frac{1}{0}]$. The next result describes all the matrices associated to triples of that form.
\begin{proposition}
\label{first description}
The matrices obtained by alternating the mutations  $\mu_{-1}$ and $\mu_0$ (beginning with $\mu_{-1}$) are of the following form:\\
\begin{center}
\begin{tabular}{c c c}
$\px \px \frac{a}{1} \px \px \px \frac{a-1}{1} \px \px  \frac{1}{0} \ $ 
&  &
$\px \px \px \frac{a}{1} \px \px \px \px \px \px \frac{a+1}{1} \px \px \px  \frac{1}{0} \px $ \vspace{3mm} \\ 
$\left(
\begin{tabular}{c c c}
 & $B^+\px $&  \vspace{1mm} \\
$1-a$ & $ a $ & $0$ \\
$ -a $ & $a+1$ & $0$ \\
$\phantom{-}0$ & $0$ & $1$ 
\end{tabular}
\right)$
&
$\px \px \px  \text{or} \px \px \px $
&
$\left(
\begin{tabular}{c c c}
 & $B^-\px $&  \vspace{1mm}  \\
$a+1$ & $ -a \px $ & $0$ \\
$ a+2 $ & $-(a+1) \px$ & $0$ \\
$0$ & $0$ & $1$
\end{tabular}
\right)$.
\end{tabular}
\end{center}
\end{proposition}

\begin{proof}
We proceed by induction on $a$. Take as base for the induction the calculus made in (\ref{alternating}). Suppose that $[\frac{a}{1},\frac{a-1}{1},\frac{1}{0}] \longrightarrow M$ satisfies the statement, mutating both in direction $-1$ we obtain:

\begin{equation*}
\begin{tabular}{c l c}
M=
$\left(
\begin{tabular}{c c c}
 & $B^+\px \px \px$& \vspace{1mm}  \\
$-a+1$ & $ a $ & $0$ \\
$ -a $ & $a+1$ & $0$ \\
$\phantom{-}0$ & $0$ & $1$ 
\end{tabular}
\right)$
&
\xymatrix{
& \ar[r]^{\mu_{-1}} & &
 }
 &
  $\left(
\begin{tabular}{c c c}
 & $B^-\px \px\px\px\px\px$& \vspace{1mm}  \\
$-a+1+2a$ & $ -a \px $ & $0$ \\
$ -a+2(a+1) $ & $-(a+1) \px$ & $0$ \\
$0$ & $0$ & $1$ 
\end{tabular}
\right)$
\end{tabular}
\end{equation*}
In this case we are done. The second case is obtained in the same way.
\end{proof}

Notice that the matrices just described are associated to a triple lying in $\T_3^{-1}$. The next theorem describes the rest of the matrices associated to the triples in $\T_3^{-1}$. All other exchange matrices will be obtained from those described in Proposition \ref{first description} and Theorem \ref{second description} below.

\begin{theorem}
\label{second description}
Let $M$ be the exchange matrix associated to a triple $[q_{0},q_{-1}, q_{\infty}]\neq [\frac{a}{1},\frac{a\pm1}{1},\frac{1}{0}]$ lying in $  \T_3^{-1}$. Consider the following six cases:
\begin{itemize}
\item[(i)] $q_0<q_{\infty}<q_{-1}$ 
\vspace{1mm}
\item[(ii)] $q_{-1}<q_{\infty}<q_{0}$
\vspace{1mm}
\item[(iii)] $q_0<q_{-1}<q_{\infty}$\vspace{1mm}
\item[(iv)]  $q_{\infty}<q_{-1}<q_0$\vspace{1mm}
\item[(v)]  $q_{-1}<q_0<q_{\infty}$\vspace{1mm}
\item[(vi)]  $q_{\infty}<q_0<q_{-1}$\vspace{1mm}
\end{itemize}
Let $q_0=\frac{a}{b},\ q_{-1}=\frac{c}{d}$ and $ q_{\infty}=\frac{e}{f}$. Then, in each case, the matrix $M$ is of the form:
\begin{itemize}
\item[(i)] 
$\left(
\begin{tabular}{c c c}
$a+1$ & $ -c+1 $ & $c-a-1$ \\
$ a+b+1 $ & $-c-d+1$ & $c+d-a-b-1$ \\
$b+1$ & $-d+1$ & $d-b-1$ 
\end{tabular}
\right)$
\vspace{4mm}
\item[(ii)] 
$
\left(
\begin{tabular}{c c c}
$-a+1$ & $ c+1 $ & $a-c-1$ \\
$ -a-b+1 $ & $c+d+1$ & $a+b-c-d-1$ \\
$-b+1$ & $d+1$ & $b-d-1$ 
\end{tabular}
\right)
$
\vspace{4mm}
\item[(iii)] 
$
\left(
\begin{tabular}{c c c}
$a+1$ & $ e-a-1 $ & $-e+1$ \\
$ a+b+1 $ & $e+f-a-b-1$ & $-e-f+1$ \\
$b+1$ & $f-b-1$ & $-f+1$ 
\end{tabular}
\right)
$
\vspace{4mm}
\item[(iv)]
$
\left(
\begin{tabular}{c c c}
$-a+1$ & $ a-e-1 $ & $e+1$ \\
$ -a-b+1 $ & $a+b-e-f-1$ & $e+f+1$ \\
$-b+1$ & $b-f-1$ & $f+1$ 
\end{tabular}
\right)
$
\vspace{4mm}
\item[(v)] 
$
\left(
\begin{tabular}{c c c}
$e-c-1$ & $ c+1 $ & $-e+1$ \\
$ e+f-c-d-1 $ & $c+d+1$ & $-e-f+1$ \\
$f-d-1$ & $d+1$ & $-f+1$ 
\end{tabular}
\right)
$
\vspace{4mm}
\item[(vi)] 
$
\left(
\begin{tabular}{c c c}
$c-e-1$ & $ -c+1$ & $e+1$ \\
$ c+d-e-f-1 $ & $-c-d+1$ & $e+f+1$ \\
$d-f-1$ & $-d+1$ & $f+1$ 
\end{tabular}
\right)
$
\end{itemize}
In cases $(i)$, $(iv)$ and $(v)$, the principal part of $M$ is $B^+$, in the other cases it is $B^-$.
\end{theorem}
\begin{proof}
First we show that the matrices obtained by mutating those of Proposition \ref{first description} in direction $\infty$ satisfy the statement:
\begin{center}
\begin{tabular}{c c c}
$\ \px \frac{a}{1} \px \px \px \frac{a-1}{1} \px \px  \frac{1}{0} $ 
&  &
$\px \px \px \frac{a}{1} \px \px \px \px \frac{a-1}{1} \px \px \px  \frac{2a-1}{2} \px \px$ \vspace{3mm} \\ 
$\left(
\begin{tabular}{c c c}
 & $B^+\px $&  \vspace{1mm} \\
$1-a$ & $ a $ & $0$ \\
$ -a $ & $a+1$ & $0$ \\
$\phantom{-}0$ & $0$ & $1$ 
\end{tabular}
\right)$
&
\xymatrix{
& \ar[r]^{\mu_{\infty}} & &
 }
&
$\left(
\begin{tabular}{c c c}
 & $B^-\px $&  \vspace{1mm}  \\
$1-a$ & $ a \px $ & $0$ \\
$ -a $ & $a+1 \px$ & $0$ \\
$0$ & $2$ & $-1 \px$ 
\end{tabular}
\right)$
\end{tabular}
\end{center}
For the matrices obtained by the assignment $[\frac{a}{1},\frac{a+1}{1},\frac{1}{0}] \longrightarrow M$, we do the same. Suppose that $[q_{-1},q_0, q_{\infty}]=\mu_{\infty}[\frac{a}{1},\frac{a\pm1}{1},\frac{1}{0}]$for some $a>0$, we proceed by induction on $c[q_{-1},q_0, q_{\infty}]$. Take as base for the induction the first part of the proof. Now, notice that the cases (for the triples) are related by mutation in the directions described in the following diagram (we only consider directions which increase the complexity):

\begin{equation*}
\xymatrix{
& (i) \ar@<-1ex>[r]_{0} \ar@<-1ex>[ld]_{-1} & (vi) \ar@<-1ex>[l]_{\infty} \ar@<-1ex>[rd]_{-1} & \\
(iii) \ar@<-1ex>[ru]_{\infty} \ar@<-1ex>[rd]_{0} & & & (iv) \ar@<-1ex>[lu]_{0} \ar@<-1ex>[ld]_{\infty}
 \\
& (v) \ar@<-1ex>[lu]_{-1} \ar@<-1ex>[r]_{\infty} & (ii)  \ar@<-1ex>[l]_{0} \ar@<-1ex>[ru]_{-1} &
}
\end{equation*}
Thus, we will be done if we prove that the matrices of the statement satisfies the same diagram. Once again, we only prove one case and in one direction, it will be obvious that all cases are similar. Consider case $(i)$, and mutate in direction 0:
\begin{center}
\begin{tabular}{c}
$\px \px \px \frac{a}{b} \px \px \px \px \px \px \px \px \px \frac{c}{d} \px \px \px \px \px \px \px \px \px \px  \frac{a+c}{b+d} \px \px \px \px $ \vspace{3mm}
\\
$\left(
\begin{tabular}{c c c}
 & $\px \px \px B^+$&  \vspace{1mm}  \\
$a+1$ & $ -c+1 $ & $c-a-1$ \\
$ a+b+1 $ & $-c-d+1$ & $c+d-a-b-1$ \\
$b+1$ & $-d+1$ & $d-b-1$ 
\end{tabular}
\right)$
\\
\xymatrix{
& \ar[d]^{\mu_{0}} \\
&
 }
\vspace{3mm} 
\\
$\ \px \frac{a+2c}{b+2d} \px \px \px \px \px \px \px  \px \frac{c}{d} \px \px \px \px \px \px \px  \px \px \px \px \px \px \px \px  \frac{a+c}{b+d}\px \px \px \px \px \px \px \px \px$ \vspace{3mm}
\\
$\left(
\begin{tabular}{c c c}
 & $\px \px \px \px \px \px \px B^-\px $&  \vspace{1mm}  \\
$-1-a$ & $ -c+1 \px $ & $c-a-1+2(a+1)$ \\
$ -a-b-1 $ & $-c-d+1 \px$ & $c+d-a-b-1+2(a+b+1)$ \\
$-b-1$ & $-d+1$ & $d-b-1+2(b+1)\px$ 
\end{tabular}
\right)$
\end{tabular}
\end{center} 
\end{proof}

We have described all the matrices associated to a triple in $\T_3^{-1}$. In order to describe the rest, we consider the action of the alternating group $A_n=\lbrace id, (123),(132) \rbrace$ on \text{Mat}$(3\times 3,\Z)$ by cyclic permutation of its ``diagonals", \ie if $B=(b_{ij})$ then $\sigma \cdot B=(b_{\sigma(i)\sigma(j)})$.

\begin{proposition}
\label{complete description}
Let $M$ be the matrix associated to a triple $[q_{0},q_{-1}, q_{\infty}]\in \T_3^{-1}$. Then $\phi[q_{0}, q_{-1} ,q_{\infty} ] \longrightarrow (123) \cdot M$ and $\psi[q_{0}, q_{-1} ,q_{\infty} ] \longrightarrow (132) \cdot M$.
\end{proposition}
\begin{proof}
Since the matrices $B^{\pm}$ and the $3\times 3$ identity matrix are invariant under the action of $(123)$ and $(132)$, and the functions $\phi$ and $\psi$ preserve the mutation in the sense of \ref{iso of components}, it is sufficient to verify the statement for the matrices associated to $\phi[\frac{0}{1}, \frac{-1}{1}, \frac{1}{0}]$ and $\psi[\frac{0}{1}, \frac{-1}{1}, \frac{1}{0}]$, but this follows at once.
\end{proof}
\begin{observation}
Each $\mathbf{c}$-vector of $\ca_M$ has either only non negative entries or only non positive entries. The corresponding fact is known for each skew-symmetric cluster algebra of geometric type, \cf section 1 of \cite{Nakanishi Zelevinsky}.
\end{observation}

In the same way as we did with the exchange matrices, we assign to each Farey Triple a $\mathbf{g}$-matrix.

\begin{theorem}
Let $M$ be the $\mathbf{g}$-matrix associated to a triple $[q_{0},q_{-1}, q_{\infty}]\in \T^{-1}_{3}$. Consider the following three cases:
\begin{itemize}
\item[(i)] $[q_{0},q_{-1}, q_{\infty}]=[\frac{a}{1},\frac{a-1}{1},\frac{1}{0}]$
\vspace{1mm}
\item[(ii)] $[q_{0},q_{-1}, q_{\infty}]=[\frac{a}{1},\frac{a+1}{1},\frac{1}{0}]$
\vspace{1mm}
\item[(iii)] $q_{0}=\frac{a}{b},\ q_{-1}=\frac{c}{d}$ and $ q_{\infty}=\frac{e}{f}\neq\frac{1}{0}$
\end{itemize}
Then, in each case, the matrix $M$ is of the form:
\begin{itemize}
\item[(i)] 
$\left(
\begin{array}{c c c}
a+1 & a & 0\\
-a & -a+1 & 0 \\
0 & 0 & 1
\end{array} \right)$
\vspace{2mm} 
\item[(ii)]
$\left(
\begin{array}{c c c}
a+1 & a+2 & 0\\
-a & -(a+1) & 0 \\
0 & 0 & 1
\end{array} \right)$
\vspace{2mm}
\item[(iii)]
$\left(
\begin{array}{c c c}
a+1 & c+1 & e+1\\
b-a-1 & d-c-1 & f-e-1 \\
1-b & 1-d & 1-f
\end{array} \right)$
\end{itemize}
\end{theorem}
\begin{proof}
In view of \cite[Theorem 1.2]{Nakanishi Zelevinsky}, we can calculate $M$ by inverting and transposing the complementary part of the exchange matrices of Proposition \ref{first description} and Theorem \ref{second description}. We can easily see that for the cases $(i)$ and $(ii)$ we are done. Now assume that $[q_{0},q_{-1}, q_{\infty}]$ is as in case (iii). Notice that all the matrices of Theorem \ref{second description} can be expressed in terms of a single matrix, namely
\begin{equation*}
\lambda \left(
\begin{array}{c c c}
e-c+cf-de & a-e+bf-af & c-a+ad-bc \\
e+f-c-d+cf-de & a+b-e-f+bf-af & c+d-a-b+ad-bc \\
f-d+cf-de & b-f+bf-af & d-b+ad-bc
\end{array} \right),
\end{equation*}
where $\lambda=(cf-de+be-af+ad-bc)^{-1}$. An straightforward calculation shows that the inverse of such matrix is
\begin{equation*}
\left(
\begin{array}{c c c}
a+1 & b-a-1 & 1-b\\
c+1 & d-c-1 & 1-d \\
e+1 & f-e-1 & 1-f
\end{array} \right).
\end{equation*}
\end{proof}

\begin{proposition}
\label{complete description}
Let $M$ be the $\mathbf{g}$-matrix associated to a triple $[q_{0},q_{-1}, q_{\infty}]\in \T_3^{-1}$. Then $\phi[q_{0}, q_{-1} ,q_{\infty} ] \longrightarrow (123) \cdot M$ and $\psi[q_{0}, q_{-1} ,q_{\infty} ] \longrightarrow (132) \cdot M$.
\end{proposition}
\begin{proof}
The operations of transposing and inverting a matrix commute with the action $A_3$. 
\end{proof}
We notice that the $\mathbf{g}$-vectors of $\ca_M$ lie in the plane $x+y+z=1$. The first vectors obtained by iterated mutation of the initial seed are represented in the Figure \ref{g vectors}. Another interpretation of this figure is given in \cite{Fock Goncharov}.
\begin{figure}[htbp]
\label{g vectors}
\includegraphics[scale=0.7]{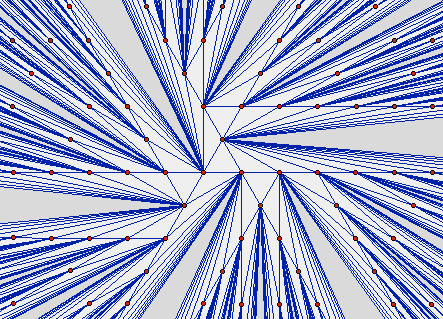}
\caption{The $\mathbf{g}$-vectors of $\ca_M$}
\end{figure}

\end{document}